\documentclass[leqno,12pt]{amsart} 
\usepackage{amssymb}
\usepackage{enumerate}  
\usepackage{graphicx}  
\usepackage{epstopdf}
\usepackage{comment}
\usepackage{array}

 \makeatletter
    
    \@addtoreset{equation}{section}
  \makeatother
\theoremstyle{plain} 
\newtheorem{theorem}{\indent\sc Theorem}[section] 

\newtheorem{corollary}[theorem]{\indent\sc Corollary}
\newtheorem{proposition}[theorem]{\indent\sc Proposition}

\theoremstyle{definition} 
\newtheorem{definition}[theorem]{\indent\sc Definition}
\newtheorem{remark}[theorem]{\indent\sc Remark}
\newtheorem{example}[theorem]{\indent\sc Example}

\begin{document}

\title{\uppercase{The Ricci curvature on directed graphs}} 
\author{\textsc{Taiki Yamada}}
\date{} 
%


\footnote{ 
2010 \textit{Mathematics Subject Classification}.
Primary 05C12; Secondary 52C99 .
}


\keywords{ 
Graph theory, Discrete differential geometry
}
\address{ 
Mathematical Institute in Tohoku University \endgraf
Sendai 980-8578 \endgraf
Japan
}
\email{mathyamada@dc.tohoku.ac.jp}



\maketitle

\begin{abstract}
In this paper, we consider the Ricci curvature of a {\em directed} graph, based on Lin-Lu-Yau's definition. We give some properties of the Ricci curvature, including conditions for a directed regular graph to be Ricci-flat. Moreover, we calculate the Ricci curvature of the cartesian product of directed graphs.
\end{abstract}


\section{Introduction} The Ricci curvature is one of the most important concepts in Riemannian geometry. In space physics, Ricci-flat manifolds represent vacuum solutions to an analogue of Einstein's equation for Riemannian manifolds with vanishing cosmological constant. They are used in the theory of general relativity. In mathematics, Calabi-Yau manifolds are Ricci-flat and can be applied to the superstring theory. There are some definitions of generalized Ricci curvature, one of which is Olivier's coarse Ricci curvature (see \cite{Ol1}, \cite{Ol2}). It is formulated by the 1-Wasserstein distance on a metric space $(X, d)$ with a random walk $m=\left\{m_{x} \right\}_{x \in X}$, where $m_{x}$ is a probability measure on $X$. The coarse Ricci curvature is defined as, for two distinct points $x, y \in X$, 
	\begin{eqnarray*}
	\kappa(x, y) := 1 - \cfrac{W(m_{x}, m_{y})}{d(x, y)},
	\end{eqnarray*}
  where $W$($m_{x}, m_{y}$) is the $1$-Wasserstein distance between $m_{x}$ and $m_{y}$.\\
　On the other hand, the graph theory is used to model many types of relations and processes in physical, biological, social and information systems (see \cite{Wa1}, \cite{Wat1} and \cite{Wu1}). A graph $G =(V, E)$ is a pair of the set $V$ of vertices and the set $E$ of edges. If each edge is represented as an ordered pair of vertices, $G$ is called a {\em directed} graph. \\
　In 2010, Lin-Lu-Yau \cite{Yau1} defined the Ricci curvature of an undirected graph by using the coarse Ricci curvature of the lazy random walk, and they studied the Ricci curvature of the product space of graphs and random graphs. They also considered the Ricci-flat graph and classified undirected Ricci-flat graphs with girth at least five (see \cite{Lin1}). In 2012, Jost and Liu \cite{Jo2} studied the relation between the Ricci curvature and the local clustering efficient. Recently, the Ricci curvature on graphs was applied to cancer network \cite{Tan}, internet topology \cite{Ni} and so on. Sometimes it seems important to consider directed graphs as networks. However, curvatures on directed graphs have not yet been discussed well because it is much more difficult than the undirected case.\\
　In this paper, we define the Ricci curvature of a {\em directed} graph based on Lin-Lu-Yau's definition, and state basic properties (\S2). For some examples, we calculate the Ricci curvature explicitly (\S3). Then giving lower and upper bounds, we obtain conditions for a directed regular graph to be Ricci-flat (Theorem \ref{main0}). Finally, we generalize it to the cartesian product graph (Theorem \ref{main1}).

\section*{acknowledgment} 
The author thanks his supervisors, Professor Reiko Miyaoka and Professor Takashi Shioya, for their continuous support and providing important comments. He also thanks the referee for his/her valuable comments and suggestions.

\section{Definition of Ricci curvature on directed graphs}
　Throughout the paper, we always assume that a graph $G=(V,E)$ is directed. If not, it will be clearly stated. For $x, y \in V$, we write $(x, y)$ as an edge from $x$ to $y$.  We denote the set of vertices of $G$ by $V(G)$ and the set of edges of $G$ by $E(G)$.
\begin{definition}
	\begin{enumerate}
  	\renewcommand{\labelenumi}{(\arabic{enumi})}
	\item A {\em directed path} from $x \in V(G)$ to vertex $y \in V(G)$ is a sequence of edges\\ $\left\{(a_{i}, a_{i+1)} \right\}_{i=0}^{n-1}$, where $a_{0} = x$, $a_{n} = y$. We call $n$ the {\em length} of the path.
   	\item The {\em distance} $d(x, y)$ between two vertices $x, y \in V$ is given by the length of a shortest directed path from $x$ to $y$.
   	\end{enumerate}
\end{definition}

\begin{remark}
If $G$ is {\em strongly connected} (i.e., there exists a directed path from $x$ to $y$ for any $x, y \in V$), then the distance is finite. The distance function satisfies positivity and triangle inequality, but not necessarily the symmetry.
\end{remark}

\begin{definition}
	\begin{enumerate}
 	\renewcommand{\labelenumi}{(\arabic{enumi})}
 	\item For any $x \in V$, the {\em out-neighborhood} of $x$ is defined as
   		\begin{eqnarray*}
   		\Gamma^{\mathrm{out}}(x) := \left\{y \in V \mid (x, y) \in E \right\}.
   		\end{eqnarray*}
  	\item For any $x \in V$, the {\em out-degree} of $x$, denoted by $d_{x}^{\mathrm{out}}$, is the number of edges starting from $x$, i.e., $d_{x}^{\mathrm{out}} = |\Gamma^{\mathrm{out}}(x)|$.
   	\item We call $G$ $d$-{\em regular graph} if every vertex has the same out-degree $d$.
   	\end{enumerate}
\end{definition}

In this paper, we assume that a directed graph $G$ has the following properties.
\begin{enumerate}
\renewcommand{\labelenumi}{(\arabic{enumi})}
\item Locally finiteness (every vertex has a finite degree)
\item Simpleness (there exist no loops and no multi-edges)
\item Strongly connectedness 
\end{enumerate}

\begin{definition}
For any $x \in V(G)$ and any $\alpha \in [0,1]$, we define a probability measure $m_{x}^{\alpha}$ on $V(G)$ by
	\begin{eqnarray*}
    	m_{x}^{\alpha}(v) = 
    		\begin{cases}
    		\alpha , &\mathrm{if}\ v = x, \\
    		\cfrac{1 - \alpha}{d^{\mathrm{out}}_{x}}, &\mathrm{if}\ (x, v) \in E, \\
    		0, & \mathrm{\mathrm{otherwise}}.
    		\end{cases}
   	\end{eqnarray*}
\end{definition} 

\begin{definition}
For two probability measures $\mu$ and $\nu$ on $V(G)$, the 1-Wasserstein distance between $\mu$ and $\nu$ is given by
	\begin{eqnarray*}
	W(\mu, \nu) = \inf_{A} \sum_{u, v \in V}A(u, v)d(u, v),
	\end{eqnarray*}
	where $A : V(G) \times V(G) \to [0, 1]$ runs over all maps satisfying 
	\begin{eqnarray}
	\label{coupling}
		\begin{cases}
		\sum_{v \in V}A(u, v) = \mu(u),\\
		\sum_{u \in V}A(u, v) = \nu(v).
		\end{cases}
	\end{eqnarray}
Such a map $A$ is called a {\em coupling} between $\mu$ and $\nu$. 
\end{definition}

\begin{remark}
To take the infimum of couplings in the definition of $W$, we should check that the set of couplings is not empty. If we take two probability measures $m_{x}^{\alpha}$ and $m_{y}^{\alpha}$ for $x, y \in V(G)$, then there always exists at least one coupling between two probability measures. In fact, we define the coupling $\bar{A}$ between $m_{x}^{\alpha}$ and $m_{y}^{\alpha}$ by
\begin{eqnarray*}
\bar{A}(u, v)=
\begin{cases}
\alpha,& \mathrm{if}\ u=x, v=y,\\
\cfrac{m_{x}^{\alpha}(u)}{d^{\mathrm{out}}_{y}},& \mathrm{if}\ u \in \Gamma^{\mathrm{out}}(x), v \in \Gamma^{\mathrm{out}}(y),\\
0, & \mathrm{\mathrm{otherwise}}.
\end{cases}
\end{eqnarray*}
It is easy to show that $\bar{A}$ satisfies \eqref{coupling}.
\end{remark}

\begin{remark}
A coupling $A$ that attains the Wasserstein distance, does not necessarily exist since the distance is not symmetry. If it exists, we call it the {\em optimal coupling}.
\end{remark}

\begin{definition}
	\label{Ricci}
For any two distinct vertices $x, y \in V$, the {\em $\alpha$-Ricci curvature} of $x$ and $y$ is defined as
   	\begin{eqnarray*}
    	\kappa_{\alpha}(x, y) =  1 - \cfrac{W(m_{x}^{\alpha}, m_{y}^{\alpha})}{d(x, y)}.
   	\end{eqnarray*}
\end{definition}

\begin{remark}[\cite{Yau1}]
	\label{concave}
For any two vertices $x$ and $y$, $\kappa_{\alpha}(x, y)$ is concave in $\alpha \in [0, 1]$. 
\end{remark}

\begin{proposition}
	\label{kantoro}
For any two vertices $x$ and $y$, we have
	\begin{eqnarray}
	\label{kan}
    	W(m_{x}^{\alpha}, m_{y}^{\alpha}) \geq \sup_{f}\left( \sum_{u \in V}f(u)m_{x}^{\alpha}(u) - \sum_{v \in V}f(v)m_{y}^{\alpha}(v) \right),
     	\end{eqnarray}
where $f: V(G) \to \mathbb{R}$ runs over all functions with $f(u) - f(v) \leq d(u, v)$.
\end{proposition}

\begin{proof}
For any coupling $A$ between $m_{x}^{\alpha}$ and $m_{y}^{\alpha}$, we have
	\begin{eqnarray*}
     		\label{duality}
      	\sum_{u, v \in V}A(u, v)d(u, v) & \geq & \sum_{u, v \in V}A(u, v)(f(u) - f(v)) \nonumber \\
      	& = & \sum_{u \in V}f(u)\sum_{v \in V}A(u, v) - \sum_{v \in V}f(v)\sum_{u \in V}A(u, v) \nonumber \\
      	& = & \sum_{u \in V}f(u)m_{x}^{\alpha}(u) - \sum_{v \in V}f(v)m_{y}^{\alpha}(v). 
     	\end{eqnarray*}
Since the left-hand side is independent of $f$, and so is the right-hand side of $A$, the proof is completed.
\end{proof}

If there exists a function satisfied the equality of \eqref{kan}, then we call it the {\em optimal function}.

\begin{remark}
Proposition \ref{kantoro} holds for $f$ running over all the 1-Lipshitz functions. In the case of undirected graphs, the equality holds in Proposition \ref{kantoro}, and we call the proposition the Kantorovich-Rubinstein duality \cite{Vi2}. 
\end{remark}

We would like to obtain the upper bound of $\kappa_{\alpha}/(1-\alpha)$. In \cite{Yau1}, this is obtained by using the symmetry of the 1-Wasserstein distance. However, in the case of directed graphs, the distance is not symmetry in general, so we use another approach. For any two distinct vertices $x, y$, we decompose $\Gamma^{\mathrm{out}}(y)$ into the following sets according to their distance from $x$ :
	\begin{eqnarray*}
     	\Gamma_{x}^{k}(y) & = & \left\{v \in \Gamma^{\mathrm{out}}(y) \mid d(x, v) = d(x, y) - k \right\},
    	\end{eqnarray*}
	where $-1 \leq k \leq d(x, y)$, since $0 \leq d(x, v) = d(x,y) - k \leq d(x, y) + d(y, v) =d(x,y)+1$. 
	
\begin{proposition}
	\label{Yamada1}
For any two distinct vertices $x, y$, we have
   	\begin{eqnarray}
   	\displaystyle \kappa_{\alpha}(x, y) \leq \frac{1 - \alpha}{d(x, y)}\left(1 + \frac{1}{d^{\mathrm{out}}_{y}} \sum_{k = -1}^{d(x, y)} k |\Gamma_{x}^{k}(y)| \right).
   	\end{eqnarray}
\end{proposition}

\begin{proof}
For a fixed $x \in V$, define $f(z) := - d(x, z)$ for $z \in V$. Then it follows that
	\begin{eqnarray*}
     	f(z) - f(w) & = & - d(x, z) + d(x, w) \\
     	& \leq & d(z, w).
     	\end{eqnarray*}
By Proposition \ref{kantoro}, we have
      	\begin{eqnarray*}
      	W(m_{x}^{\alpha}, m_{y}^{\alpha}) & \geq & \sum_{z \in V}d(x, z)(m_{y}^{\alpha}(z) - m_{x}^{\alpha}(z)) \\
      	& = & \alpha d(x, y) + \sum_{k = -1}^{d(x, y)}(d(x, y) -  k)|\Gamma_{x}^{k}(y)|\frac{1 - \alpha}{d^{\mathrm{out}}_{y}} - (1 - \alpha)\\
      	& = &  \alpha d(x, y) +  d(x, y) (1 - \alpha)  - \cfrac{1 - \alpha}{d^{\mathrm{out}}_{y}} \sum_{k = -1}^{d(x, y)}k |\Gamma_{x}^{k}(y)| - (1 - \alpha) \\
      	& = & d(x, y) - \cfrac{1 - \alpha}{d^{\mathrm{out}}_{y}} \sum_{k = -1}^{d(x, y)}k |\Gamma_{x}^{k}(y)|  - (1 - \alpha).
     	\end{eqnarray*}
The proof is completed.
\end{proof}

\begin{corollary}
	\label{cor0}
If any edge $(x, y)$ satisfies $(y, x) \notin E$ and
\begin{eqnarray*}
  	 \Gamma^{\mathrm{out}}(x) \cap \Gamma^{\mathrm{out}}(y) = \emptyset,
\end{eqnarray*}
then we have
\begin{eqnarray*}
\kappa_{\alpha}(x, y) \leq 0.
\end{eqnarray*}
\end{corollary}

\begin{proof}
If we take any edge $(x, y)$, then we have
	\begin{eqnarray*}
  	\kappa_{\alpha}(x, y) \leq 1 + \cfrac{|\Gamma^{1}_{x}(y)| - |\Gamma^{-1}_{x}(y)|}{d_{y}^{\mathrm{out}}},
 	\end{eqnarray*}
by Proposition \ref{Yamada1}. The first assumption implies $|\Gamma^{1}_{x}(y)| =0$, and the second implies $|\Gamma^{0}_{x}(y)| =0$. So, the out-degree of $y$ is equal to $|\Gamma^{-1}_{x}(y)|$. Thus, we obtain
 	\begin{eqnarray*}
 	\label{upper}
 	\kappa_{\alpha}(x, y) \leq 0.
 	\end{eqnarray*}
\end{proof}

\begin{remark}
Remark \ref{concave} implies that $h(\alpha) := \kappa_{\alpha}(x, y)/(1 - \alpha)$ is a monotone increasing function in $\alpha \in [0, 1)$ (the detail is written in the proof of Lemma 2.1 in \cite{Yau1}). Proposition \ref{Yamada1} implies that $h(\alpha)$ is bounded. Thus, the limit $\kappa(x, y) = \lim_{\alpha \to 1} \kappa_{\alpha}(x, y) / (1 - \alpha)$ exists.
\end{remark}

\begin{definition}
For any two distinct vertices $x, y \in V$, the {\em Ricci curvature} of $x$ and $y$ is defined as
	\begin{eqnarray*}
    	\kappa(x, y) = \lim_{\alpha \to 1}\cfrac{\kappa_{\alpha}(x, y)}{1 - \alpha}.
   	\end{eqnarray*}
   	\end{definition}
	
   	Whenever we are interested in the lower bound of Ricci curvature, the following lemma implies that it is sufficient to consider the Ricci curvature of the edge, although the Ricci curvature is defined for any pair of vertices.
\begin{proposition}
If $\kappa(u, v) \geq \kappa_{0}$ for any edge $(u, v) \in E(G)$, then $\kappa(x, y) \geq \kappa_{0}$ for any pair of vertices $(x, y)$.
\end{proposition}
The proof is similar to the case of undirected graphs \cite{Yau1} and is omitted. If $\kappa(x, y) = r \in \mathbb{R}$ holds for all edges $(x, y) \in E$, then we say that $G$ is a {\em graph of constant Ricci curvature}, and write $\kappa(G) = r$. If $\kappa(G) = 0$, we say $G$ is {\em Ricci-flat}.

\begin{remark}
On a Ricci-flat graph, $\kappa(x, y) = 0$ does not necessarily hold for any vertices $x$, $y$ unless $(x, y) \in E$. Example \ref{cycle} below is one of such examples.
\end{remark}


\section{Examples}
In this section, we calculate the Ricci curvature on some directed graphs.

\begin{definition}
For a finite graph $G$, let $M=(m_{ij})$ be the matrix defined by the following :
	\begin{eqnarray*}
   	m_{ij}=
   		\begin{cases}
    		1, &\mathrm{if}\ (v_{i}, v_{j}) \in E,\\
    		0, &\mathrm{if}\ (v_{i}, v_{j}) \not \in E,\\
   	\end{cases}
\end{eqnarray*}
where $V(G)=\left\{v_{1}, v_{2}, \cdots, v_{n} \right\}$. We call $M$ the \emph{adjacency matrix} of the graph $G$. Note that the adjacency matrix is not necessarily symmetric.
\end{definition}

\begin{example}[Complete graph $K_{2n+1}$]
\label{complete}
We consider a directed complete graph with the following adjacency matrix $M_{2n+1}$ :
	\begin{eqnarray*}
    		\begin{cases}
     		m_{1,j}=1, & j \in \left\{2, \cdots, n+1 \right\},\\
     		m_{1,j}=0, & j \in \left\{ 1 \right\} \cup \left\{n+2, \cdots, 2n+1 \right\} ,\\
     		m_{i, j}=1, & i \in \left\{2, \cdots, n+1 \right\}, j \in \left\{1+i, \cdots, n+i \right\},\\
     		m_{i, j}=0, & i \in \left\{2, \cdots, n+1 \right\}, j \not \in \left\{1+i, \cdots, n+i \right\},\\
     		m_{i, j}=1, & i \in \left\{n+2, \cdots, 2n \right\}, j \in \left\{1+i, \cdots, 2n+1 \right\} \cup \left\{1, \cdots, i-n-1 \right\},\\
     		m_{i, j}=0, & i \in \left\{n+2, \cdots, 2n \right\}, j \not \in \left\{1+i, \cdots, 2n+1 \right\} \cup \left\{1, \cdots, i-n-1 \right\},\\
     		m_{2n+1, j}=1, & j \in \left\{1, \cdots, n \right\}, \\
     		m_{2n+1, j}=0, & j \in \left\{n+1, \cdots, 2n+1 \right\}.
    		\end{cases}
   	\end{eqnarray*}
For instance, $M_{3}$, $M_{5}$, $M_{7}$ are given by
      	\begin{eqnarray*}
    	M_{3}=\left( \begin{array}{ccc}
	0 & 1 & 0  \\
	0 & 0 & 1  \\
	1 & 0 & 0  \\
	\end{array} \right),
    	M_{5}=\left( \begin{array}{ccccc}
	0 & 1 & 1 & 0 & 0 \\
	0 & 0 & 1 & 1 & 0 \\
	0 & 0 & 0 & 1 & 1 \\
	1 & 0 & 0 & 0 & 1 \\
	1 & 1 & 0 & 0 & 0 \\
	\end{array} \right),
    	M_{7}=\left( \begin{array}{ccccccc}
	0 & 1 & 1 & 1 & 0 & 0 & 0 \\
	0 & 0 & 1 & 1 & 1 & 0 & 0 \\
	0 & 0 & 0 & 1 & 1 & 1 & 0 \\
	0 & 0 & 0 & 0 & 1 & 1 & 1 \\
	1 & 0 & 0 & 0 & 0 & 1 & 1 \\
	1 & 1 & 0 & 0 & 0 & 0 & 1 \\
	1 & 1 & 1 & 0 & 0 & 0 & 0 \\
	\end{array} \right).
   	\end{eqnarray*}
	By the definition of $M_{2n+1}$, for $j \in \left\{2, \cdots, n \right\}$, we have
	\begin{eqnarray*}
	\Gamma^{\mathrm{out}}(v_{1}) \cap \Gamma^{\mathrm{out}}(v_{j}) = \left\{v_{j+1}, \cdots, v_{n+1} \right\}
	\end{eqnarray*}
	and $\Gamma^{\mathrm{out}}(v_{1}) \cap \Gamma^{\mathrm{out}}(v_{n+1}) = \emptyset$. For simplicity, we take the vertex $v_{1}$, and calculate the Ricci curvature on the edges from $v_{1}$. For $j \in \left\{2, \cdots, n+1 \right\}$ and $k \in \left\{1, \cdots, j-1 \right\}$, we define a coupling $A_{j}$ between $m^{\alpha}_{v_{1}}$ and $m^{\alpha}_{v_{j}}$ by
	\begin{eqnarray*}
	\label{coupling-complete}
	A_{j}(u,v)=
	\begin{cases}
	\alpha, & \mathrm{if}\ u=v_{1}, v=v_{j},\\
	\cfrac{1-\alpha}{n}, & \mathrm{if}\ u=v \in \Gamma^{\mathrm{out}}(v_{1}) \cap \Gamma^{\mathrm{out}}(v_{j}),\\
	\cfrac{1-\alpha}{n}, & \mathrm{if}\ u =v_{1+k}, v=v_{n+1+k},\\
	0, & \mathrm{otherwise},
	\end{cases}
	\end{eqnarray*}
	and define a function $f_{j} : \Gamma^{\mathrm{out}}(v_{1}) \to \Gamma^{\mathrm{out}}(v_{j})$ by
	\begin{eqnarray*}
	\label{function-complete}
	f_{j}(u)=
	\begin{cases}
	1, & \mathrm{if}\ u = v_{1},\\
	-1, & \mathrm{if}\ u \in \Gamma^{\mathrm{out}}(v_{j}) \setminus \Gamma^{\mathrm{out}}(v_{1}),\\
	0, & \mathrm{otherwise}.
	\end{cases}
	\end{eqnarray*}
	By using these coupling and function, for any edge $(x, y) \in E(K_{2n+1})$, the value is either one of the following.
	\begin{eqnarray*}
	\kappa(x, y) \in
	\left\{ 0,\ \cfrac{1}{n},\ \cdots ,\ \cfrac{n - 1}{n} \right\}.
	\end{eqnarray*}
\end{example}

\begin{example}[Cycle $C_{n}$]
	\label{cycle}
	We consider a directed cycle as follows.\\
Let $V(C_{n}) = \left\{x_{1}, x_{2}, \cdots, x_{n} \right\}$. For any $i \in \left\{1, 2, \cdots, n-1 \right\}$, let $(x_{i}, x_{i + 1}) \in E(C_{n})$ and $(x_{n}, x_{1}) \in E(C_{n})$. This cycle is called a {\em directed cycle}.
Then this is Ricci-flat, namely, 
   	\begin{eqnarray*}
    	\kappa(C_{n}) = 0.
   	\end{eqnarray*}
	\noindent
	However, in the middle of Figure \ref{Tree}, $\kappa (x_{1}, x_{5})=5/4$, and is not zero. 
\end{example}

\begin{example}[Tree $T$]
In general, a tree has no strongly connected direction. However, if we consider the directed tree with  $d_{v}^{\mathrm{out}} = 1$ for any $v \in V(T)$, we can calculate the Ricci curvature on any edges. The Ricci curvature is given by
	\begin{eqnarray*}
    	\kappa(T) = 0.
    	\end{eqnarray*} 
\end{example}
     
  \begin{figure}[h]
     \begin{center} 
     \includegraphics[scale=0.45]{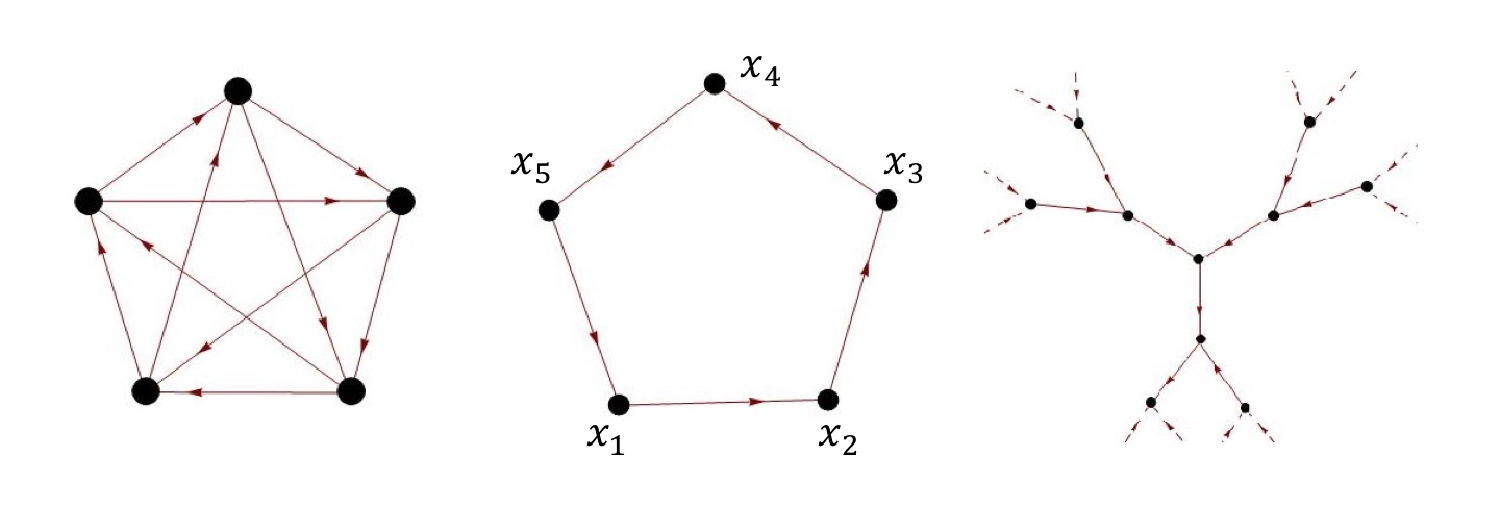}
     \end{center}
     \caption{Complete graph $K_{5}$, directed cycle $C_{5}$, and regular tree $T_{3}$}
      \label{Tree} 
    \end{figure}
    
\begin{example}[Ladder graph]
We consider an infinite graph $G$, called a {\em Ladder graph}, that is directed as follows.\\
Let $V(G)=\left\{x_{1}, x_{2}, \cdots, x_{n}, x_{n+1}, \cdots \right\}$ and
 	\begin{eqnarray*}
  	E(G) = \left\{(x_{1}, x_{2}), (x_{4}, x_{3}), (x_{6}, x_{5}), \cdots \right\} \cup \left\{ (x_{3}, x_{1}), (x_{5}, x_{3}), \cdots \right\} \cup \left\{(x_{2}, x_{4}), (x_{4}, x_{6}), \cdots \right\}.
  	\end{eqnarray*}
Then, the Ricci curvature is given by
  	\begin{eqnarray*}
   	\kappa(u, v) \leq 0\ \mathrm{for\ any}\ (u, v) \in E(G).
  	\end{eqnarray*}
\end{example} 

   \begin{figure}[h]
     \begin{center} 
     \includegraphics[scale=0.5]{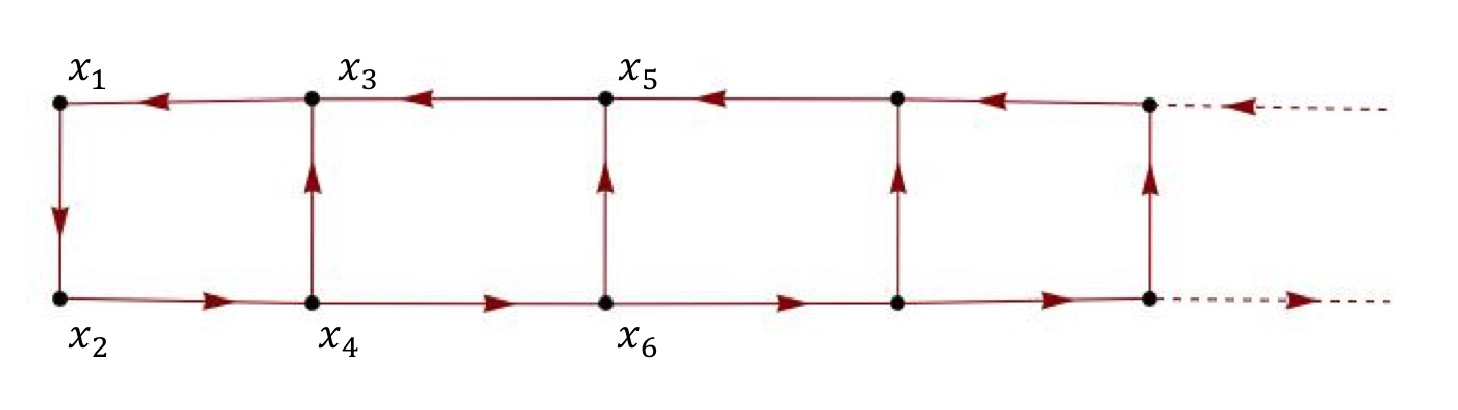}
     \end{center}
     \caption{Ladder graph}
      \label{infinite} 
    \end{figure}
    

\section{Properties of Ricci curvature on a directed graph}
In this section, we prove some properties of the Ricci curvature. 

	\subsection{Conditions to be Ricci-flat graph.}
	
\begin{proposition}
	\label{Yamada2}
For any edge $(x, y) \in E(G)$, we have
	\begin{eqnarray*}
	\kappa(x, y) \geq (1 - D)\left(1 - \cfrac{1}{d^{\mathrm{out}}_{x}} \right),
	\end{eqnarray*}
where $D := \max_{u \in \Gamma^{\mathrm{out}}(x), v \in \Gamma^{\mathrm{out}}(y)} d(u, v)$.
\end{proposition}

\begin{proof}
We take any edge $(x, y)$, and calculate Ricci curvature of $x$ and $y$ by the coupling between $m^{\alpha}_{x}$ and $m^{\alpha}_{y}$. Our transfer plan moving $m^{\alpha}_{x}$ to $m^{\alpha}_{y}$ should be as follows :
	\begin{enumerate}
	\item Move the mass of $\alpha$ from $x$ to $y$. The distance is $1$.
	\item Move the mass of $\cfrac{1 - \alpha}{d^{\mathrm{out}}_{x}}$ from $y$ to $\Gamma^{\mathrm{out}}(y)$. The distance is $1$.
	\item Fill gaps using the mass at $\Gamma^{\mathrm{out}}(x) \setminus \left\{y \right\}$.\\
The distance is at most $D:=\max_{u \in \Gamma^{\mathrm{out}}(x), v \in \Gamma^{\mathrm{out}}(y)} d(u, v)$.
	\end{enumerate}
By this transfer plan, calculating the 1-Wasserstein distance between $m_{x}^{\alpha}$ and $m_{y}^{\alpha}$, we have
	\begin{eqnarray*}
 	W(m_{x}^{\alpha}, m_{y}^{\alpha}) &\leq& \alpha + \cfrac{1 - \alpha}{d^{\mathrm{out}}_{x}} + \cfrac{D(1 - \alpha)}{d^{\mathrm{out}}_{x}}(d^{\mathrm{out}}_{x}-1)\\
 	&=& \alpha + D(1- \alpha) +\cfrac{(1 - D)(1 - \alpha)}{d^{\mathrm{out}}_{x}}.
	\end{eqnarray*}
Then we obtain
	\begin{eqnarray*}
	\kappa_{\alpha}(x, y) \geq (1 - \alpha) ( 1 - D) - \cfrac{(1 - D)(1 - \alpha)}{d^{\mathrm{out}}_{x}},
	\end{eqnarray*}
which implies
	\begin{eqnarray*}
	\kappa(x, y) \geq (1 - D)\left( 1 - \cfrac{1}{d^{\mathrm{out}}_{x}} \right).
	\end{eqnarray*}
\end{proof}

By using Proposition \ref{Yamada1} and Proposition \ref{Yamada2}, we obtain the following :

\begin{corollary}
For an edge $(x, y) \in E$, we assume that $d^{\mathrm{out}}_{x}=1$ and $(y, x) \notin E$. Then we have
	\begin{eqnarray*}
	\kappa(x, y) = 0.
	\end{eqnarray*}
\end{corollary}

\begin{proposition}
\label{cor1}
If there exists a bijective map $\phi$ : $\Gamma^{out}(x) \to \Gamma^{out}(y)$ with $d(u, \phi(u)) = 1$ for any edge $(x, y)$ and $u \in \Gamma^{\mathrm{out}}(x)$, then we have
\begin{eqnarray*}
\kappa(x, y) \geq 0.
\end{eqnarray*}
\end{proposition}

\begin{proof}
We take any edge $(x, y)$, and assume that $|\Gamma^{out}(x)|=d$. By the assumption, $G$ is a $d$-regular graph. We define a coupling $A_{0}$ between $m^{\alpha}_{x}$ and $m^{\alpha}_{y}$ by
 	\begin{eqnarray*}
 	A_{0}(u, v)=
 		\begin{cases}
 		\alpha, & \mathrm{if}\ u=x, v=y,\\
		\cfrac{1 - \alpha}{d}, & \mathrm{if}\ u \in \Gamma^{\mathrm{out}}(x), v = \phi(u),\\
		0, & \mathrm{otherwise}.
 		\end{cases}
 	\end{eqnarray*}
By using this coupling, calculating the 1-Wasserstein distance between $m_{x}^{\alpha}$ and $m_{y}^{\alpha}$, we have
	\begin{eqnarray*}	
	W(m^{\alpha}_{x}, m^{\alpha}_{y}) &\leq& \alpha + \sum_{i=1}^{d}\cfrac{1 - \alpha}{d}\\
	&=& \alpha + (1 - \alpha) = 1.
	\end{eqnarray*}
Thus we obtain
	\begin{eqnarray*}
	\kappa(x, y) \geq 0.
	\end{eqnarray*}
\end{proof}

Combining Corollary \ref{cor0} and Proposition \ref{cor1}, we obtain the following :

\begin{theorem}
	\label{main0}
Assume that any edge $(x, y)$ satisfies the following conditions : 
 	\begin{enumerate}
  	\item $(y, x) \notin E$, and $\Gamma^{\mathrm{out}}(x) \cap \Gamma^{\mathrm{out}}(y) = \emptyset$.
  	\item There exists a bijective map $\phi$ : $\Gamma^{out}(x) \to \Gamma^{out}(y)$ with $d(u, \phi(u)) = 1$ for any $u \in \Gamma^{\mathrm{out}}(x)$.
 	\end{enumerate}
Then $G$ is Ricci-flat.
\end{theorem}

\begin{remark}
\label{some}
We cannot replace ``any" by ``some" in the condition (1) of Theorem \ref{main0}. In fact, under the condition (2), we have examples with $(y,x)\in E$ for some edge, but not all edges (Figure \ref{counter}(a)), and also, $\Gamma^{\mathrm{out}}(x) \cap \Gamma^{\mathrm{out}}(y) = \emptyset$ for some $(x, y)$, but not all $(x, y)$ (Figure \ref{counter} (b)). In fact, the Ricci curvature of $(x_{0}, z)$ in Figure \ref{counter} (b) is $1/2$, not zero. On the other hand, the graph in Figure \ref{counter} (a) is Ricci-flat, so there exists a Ricci-flat graph that does not satisfy the conditions of Theorem \ref{main0}.
\end{remark}

\begin{remark}
\label{need}
Both conditions of Theorem \ref{main0} are needed. In fact, the graph in Figure \ref{counter} (c) satisfies (1),but the distance from $z \in \Gamma^{\mathrm{out}}(x_{0})$ to any vertex in $\Gamma^{\mathrm{out}}(y_{0})$ is 2. So there do not exist bijective maps satisfying (2). On the other hand, the complete graph $K_{2n+1}$ (Example \ref{complete}) satisfies only the condition (2).
\end{remark}

   \begin{figure}[h]
     \begin{center} 
     \includegraphics[scale=0.45]{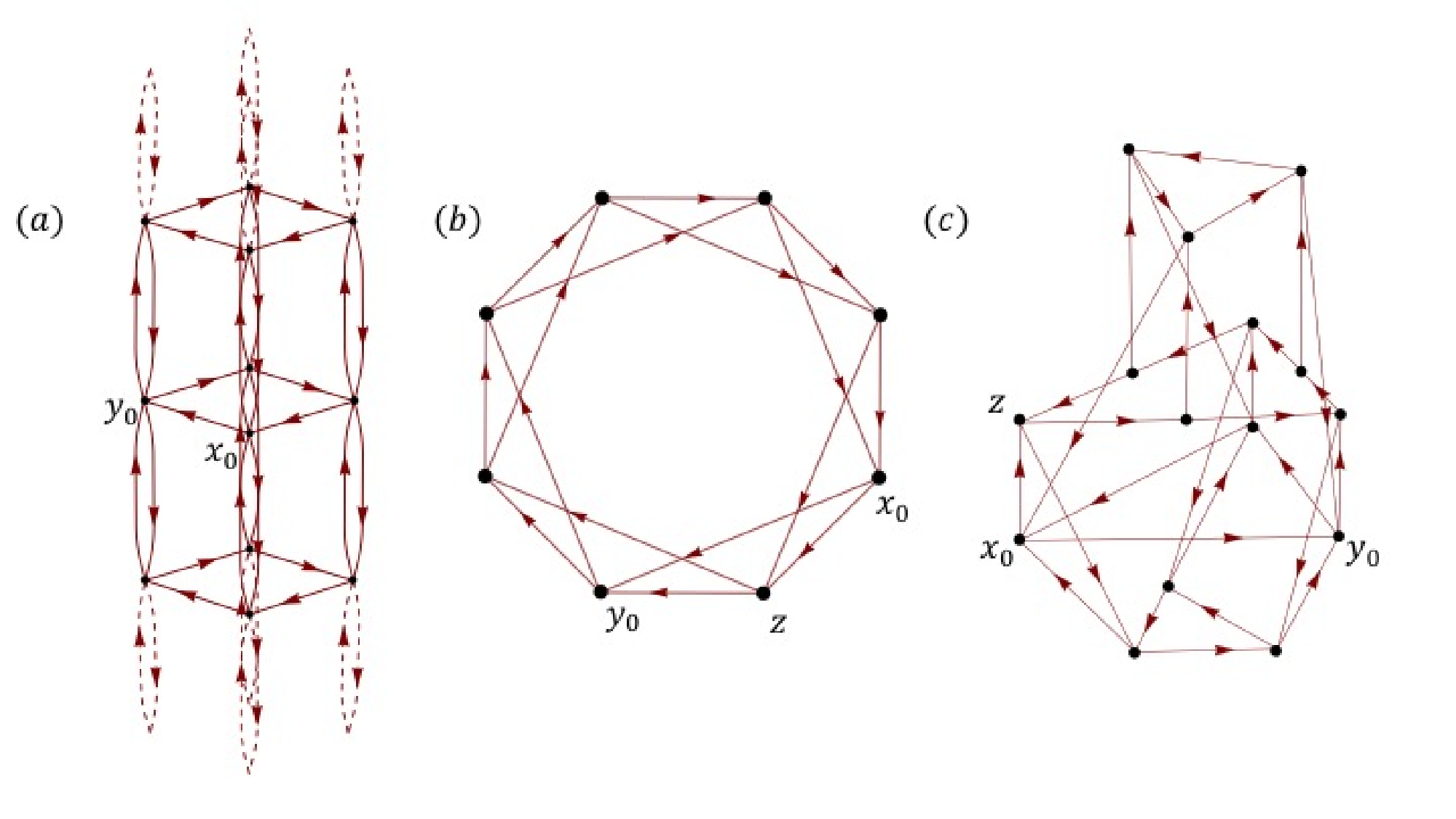}
     \end{center}
     \caption{Example of Remark \ref{some}} and Remark \ref{need}
      \label{counter} 
    \end{figure}

\subsection{Cartesian product graph}

\begin{definition}
For two directed graphs $G = (V(G), E(G))$ and $H = (V(H), E(H))$, the {\em cartesian product graph} of $G$ and $H$, denoted by $G \times H$, is a directed graph over the vertex set $V(G \times H)=V(G) \times V(H)$, and $(x_{1}, y_{1})$, $(x_{2}, y_{2}) \in V(G \times H)$ are
connected if 
 	\begin{eqnarray*}
 	x_{1} = x_{2} \ \mathrm{and}\ (y_{1}, y_{2}) \in E(H),
 	\end{eqnarray*}
or
 	\begin{eqnarray*}
 	(x_{1}, x_{2}) \in E(G) \ \mathrm{and}\ y_{1} = y_{2}.
 	\end{eqnarray*}
\end{definition}

\begin{remark}
If $G$ is a $d_{G}$-regular graph and $H$ is a $d_{H}$-regular graph, then $G \times H$ is a $(d_{G} + d_{H})$-regular graph.
\end{remark}

For any graph $G$, the Ricci curvature on $G$ is denoted by $\kappa^{G}$. When we calculate the Ricci curvature on the cartesian product graph, we need to consider an optimal coupling and an optimal function. The details of the calculation are written in Theorem 3.1 in \cite{Yau1}.

\begin{theorem}
	\label{main1}
Assume that $G$ satisfies the conditions of Theorem \ref{main0}. Then for any $d_{H}$-regular graph $H$, we have
 	\begin{eqnarray*}
 	\kappa^{G \times H}((x_{1}, y),( x_{2}, y)) = 0,
 	\end{eqnarray*}
for $(x_{1}, x_{2}) \in E(G)$ and $y \in V(H)$.
\end{theorem}

\begin{proof}
We take any edge $((x_{1}, y),( x_{2}, y)) \in E(G \times H)$, and give the coupling between $m_{(x_{1}, y)}^{\alpha}$ and $m_{(x_{2}, y)}^{\alpha}$. Since $G$ satisfies the conditions of Theorem \ref{main0}, by Proposition \ref{cor0}, there exists the optimal function $f(z)=-d(x_{1}, z)$ for $z \in V(G)$, and by Proposition \ref{cor1}, $G$ is a $d_{G}$-regular graph and the optimal coupling $A_{0}$ is given in the proof of Proposition \ref{cor1}, i.e.,
	\begin{eqnarray*}
	1 = \sum_{z \in V}d(x_{1}, z)(m_{x_{2}}^{\alpha}(z) - m_{x_{1}}^{\alpha}(z)) 
	\leq W(m_{x_{1}}^{\alpha}, m_{x_{2}}^{\alpha}) 
	\leq \sum_{u, v \in V}A_{0}(u, v)d(u, v) = 1.
	\end{eqnarray*}
By using this coupling $A_{0}$, a map $B: V(G \times H) \times V(G \times H) \to [0,1]$ is defined by
 	\begin{eqnarray*}
 	&B((u_{1}, v_{1}), (u_{2}, v_{2}))&\\
 	&=& 
 		\begin{cases}
		\cfrac{d_{G} A_{0}(x_{1}, x_{2})}{d_{G}+d_{H}} + \cfrac{\alpha d_{H}}{d_{G}+ d_{H}}, & \mathrm{if}\ u_{1}= x_{1}, u_{2}= x_{2}, v_{1}=v_{2}=y,\\
		\cfrac{d_{G} A_{0}(u_{1}, u_{2})}{d_{G}+ d_{H}}, & \mathrm{if}\ u_{1} \in \Gamma^{\mathrm{out}}(x_{1}), u_{2} = \phi(u_{1}), v_{1}=v_{2}=y,\\
		\cfrac{d_{H} m_{y}^{\alpha}(v)}{d_{G}+d_{H}}, & \mathrm{if}\ u_{1}=x_{1}, u_{2}=x_{2}, v_{1}=v_{2}=v \in \Gamma^{\mathrm{out}}(y),\\
		0, & \mathrm{otherwise},
		\end{cases}
 	\end{eqnarray*}
that is, 
 	\begin{eqnarray*}
 	B((u_{1}, v_{1}), (u_{2}, v_{2}))= 
 		\begin{cases}
		\alpha, & \mathrm{if}\ u_{1}= x_{1}, u_{2}= x_{2}, v_{1}=v_{2}=y,\\
		\cfrac{1- \alpha}{d_{G}+ d_{H}}, & \mathrm{if}\ u_{1} \in \Gamma^{\mathrm{out}}(x_{1}), u_{2} = \phi(u_{1}), v_{1}=v_{2}=y,\\
		\cfrac{1 - \alpha}{d_{G}+d_{H}}, & \mathrm{if}\ u_{1}=x_{1}, u_{2}=x_{2}, v_{1}=v_{2}=v \in \Gamma^{\mathrm{out}}(y),\\
		0, & \mathrm{otherwise}.
		\end{cases}
 	\end{eqnarray*}
It is easy to check that this map is a coupling between $m_{(x_{1}, y)}^{\alpha}$ and $m_{(x_{2}, y)}^{\alpha}$. So, the 1-Wasserstein distance satisfies
 	\begin{eqnarray*}
 	W(m^{\alpha}_{(x_{1}, y)}, m^{\alpha}_{(x_{2}, y)}) \leq \cfrac{d_{G}}{d_{G}+ d_{H}}W(m^{\alpha}_{x_{1}}, m^{\alpha}_{x_{2}}) + \cfrac{d_{H}}{d_{G}+ d_{H}}.
 	\end{eqnarray*}
Then we have
 	\begin{eqnarray*}
 	\kappa^{G\times H}_{\alpha}((x_{1}, y), (x_{2}, y)) \geq \cfrac{d_{G}}{d_{G}+ d_{H}}\kappa^{G}_{\alpha}(x_{1}, x_{2}).
 	\end{eqnarray*}
Thus, we obtain
 	\begin{eqnarray}
 	\label{product1}
 	\kappa^{G\times H}((x_{1}, y), (x_{2}, y)) \geq \cfrac{d_{G}}{d_{G}+ d_{H}}\kappa^{G}(x_{1}, x_{2}).
 	\end{eqnarray}
　On the other hand, by using the optimal function, we define a function $F: V(G\times H) \times V(G \times H) \to \mathbb{R}$ by
	\begin{eqnarray*}
	F(u, v)=
		\begin{cases}
		f(u), & \mathrm{if}\ v= y,\\
		\cfrac{f(x_{1})+f(x_{2})+1}{2}, & \mathrm{if}\ u=x_{1}, v \neq y,\\
		\cfrac{f(x_{1})+f(x_{2}) -1}{2}, & \mathrm{if}\ u=x_{2}, v \neq y,\\
		0, & \mathrm{otherwise},
		\end{cases}
	\end{eqnarray*}
that is, 
 	\begin{eqnarray*}
 	F(u, v)=
 		\begin{cases}
		-d(x_{1}, u), & \mathrm{if}\ v= y,\\
		-1, & \mathrm{if}\ u=x_{2}, v \neq y,\\
		0, & \mathrm{otherwise}.
		\end{cases}
 	\end{eqnarray*}
It is easy to check that $F$ satisfies $F(u_{1}, v_{1}) - F(u_{2}, v_{2}) \leq d((u_{1}, v_{1}), (u_{2}, v_{2}))$. By Proposition \ref{kantoro}, the 1-Wasserstein distance satisfies
 	\begin{eqnarray*}
 	W(m^{\alpha}_{(x_{1}, y)}, m^{\alpha}_{(x_{2}, y)}) \geq \cfrac{d_{G}+\alpha d_{H}}{d_{G}+ d_{H}}W(m^{\alpha'}_{x_{1}}, m^{\alpha'}_{x_{2}}) + (1 - \alpha) \cfrac{d_{H}}{d_{G}+ d_{H}},
 	\end{eqnarray*}
where $\alpha' = \alpha(d_{G} + d_{H})/(d_{G}+\alpha d_{H})$. Then we have
 	\begin{eqnarray*}
 	\kappa^{G\times H}_{\alpha}((x_{1}, y), (x_{2}, y)) \leq \cfrac{d_{G} + \alpha d_{H}}{d_{G}+ d_{H}}	\kappa^{G}_{\alpha'}(x_{1}, x_{2}).
 	\end{eqnarray*}
Since $(1-\alpha')/(1-\alpha)=d_{G}/(d_{G}+ \alpha d_{H})$, we obtain
 	\begin{eqnarray}
 	\label{product2}
 	\kappa^{G\times H}((x_{1}, y), (x_{2}, y)) &=& \lim_{\alpha \to 1}\cfrac{1}{1-\alpha} \kappa^{G\times H}_{\alpha}((x_{1}, y), (x_{2}, y)) \nonumber \\
 	&\leq& \cfrac{d_{G}}{d_{G}+d_{H}}\kappa^{G}(x_{1}, x_{2}).
  	\end{eqnarray}
By \eqref{product1} and \eqref{product2}, we have
 	\begin{eqnarray*}
 	\kappa^{G\times H}((x_{1}, y), (x_{2}, y)) = \cfrac{d_{G}}{d_{G}+d_{H}}\kappa^{G}(x_{1}, x_{2}).
 	\end{eqnarray*}
	Since $G$ is Ricci-flat by Theorem \ref{main0}, the proof is completed.
\end{proof}

\begin{corollary}
	\label{cor_regular}
Assume that $H$ satisfies the conditions of Theorem \ref{main0}. Then for any $d_{G}$-regular graph $G$, we have
 	\begin{eqnarray*}
 	\kappa^{G \times H}((x, y_{1}),( x, y_{2})) = 0,
 	\end{eqnarray*}
for $x \in V(G)$ and $(y_{1}, y_{2}) \in E(H)$.
\end{corollary}
Combining Theorem \ref{main1} and Corollary \ref{cor_regular}, we obtain the following :
\begin{corollary}
Assume that both $G$ and $H$ satisfy the conditions of Theorem \ref{main0}. Then $G \times H$ is Ricci-flat.
\end{corollary}

\end{document}